\documentclass[12pt]{article}

\usepackage{graphicx} 
\usepackage{url}      
\newcommand{\doi}[1]{\url{http://dx.doi.org/#1}}
\usepackage{amsmath}  
\usepackage{amssymb}
\usepackage{mathrsfs}
\usepackage{multirow}

\usepackage{xspace,algorithm,algorithmic,tabularx}

\usepackage{setspace}
\usepackage{color}
\usepackage{pdflscape}
\usepackage{graphics}
\usepackage{rotating}
\usepackage{verbatim}
\usepackage{amsmath}
\usepackage{amssymb}
\usepackage{amsfonts}
\usepackage{mathrsfs}
\usepackage{array}
\usepackage{epic,eepic}
\usepackage{fancyhdr}
\usepackage{url}
\usepackage{subfigure}

\newtheorem{theorem}{Theorem}[section]
\newtheorem{definition}{Definition}[section]
\newtheorem{remark}{Remark}[section]
\newtheorem{lemma}{Lemma}[section]
\newtheorem{corollary}{Corollary}[section]

\newenvironment{proof}{\begin{trivlist}
    \item[\hskip\labelsep{\it Proof.}]}{$\hfill\Box$ \end{trivlist}}

\newcommand{\Vol}{{\rm Vol}}
\newcommand{\rd}{d}

\newcommand{\cS}{[s]}

\newcommand{\real}{\mathbb{R}}

\def\bsa{\boldsymbol{a}}
\def\bsp{\boldsymbol{p}}
\def\bsx{\boldsymbol{x}}
\def\bsy{\boldsymbol{y}}

\def\bsk{\boldsymbol{k}}
\def\bsl{\boldsymbol{l}}
\def\bsh{\boldsymbol{h}}
\def\bs0{\boldsymbol{0}}
\def\bsq{\boldsymbol{q}}

\def\bsp{\boldsymbol{p}}

\def\bsgamma{\boldsymbol{\gamma}}

\def\bsmu{\boldsymbol{\mu}}
\newcommand{\uu}{\mathfrak{u}}

\def\cardu{\vert \uu \vert}

\def\nn{\mathbb{N}}
\def\zz{\mathbb{Z}}

\newcommand{\icomp}{\mathtt{i}}

\newcommand{\Pcal}{\mathcal{P}}

\newcommand{\rr}{\mathbb{R}}
\newcommand{\cc}{\mathbb{C}}
\newcommand{\I}{\mathbf{1}}

\newcommand{\wal}{{\rm wal}}
\newcommand{\Dcal}{\mathscr{D}}

\newcommand{\Var}{{\rm Var}}
\newcommand{\diag}{{\rm diag}}

\allowdisplaybreaks

\title{A Construction of Polynomial Lattice Rules with Small Gain Coefficients}

\author{Jan Baldeaux\thanks{The support of the Australian Research Council under its Centre of Excellence Program is gratefully acknowledged.} and Josef Dick\thanks{The support of the Australian Research Council under its Centre of Excellence Program is gratefully acknowledged. The author is supported by an Australian Research Council Queen Elizabeth II Research Fellowship.}}

\date{}

\begin{document}
\maketitle

\begin{abstract}
In this paper we construct polynomial lattice rules which have, in
some sense, small gain coefficients using a component-by-component
approach. The gain coefficients, as introduced by Owen, indicate to
what degree the method improves upon Monte Carlo. We show that the
variance of an estimator based on a scrambled polynomial lattice
rule constructed component-by-component decays at a rate of
$N^{-(2\alpha + 1) +\delta}$, for all $\delta >0$, assuming that the
function under consideration has bounded variation of order $\alpha$
and where $N$ denotes the number of quadrature points. An analogous
result is obtained for Korobov polynomial lattice rules. It is also
established that these rules are almost optimal for the function
space considered in this paper. Furthermore, we discuss the
implementation of the component-by-component approach and show how
to reduce the computational cost associated with it. Finally, we
present numerical results comparing scrambled polynomial lattice
rules and scrambled digital nets.
\end{abstract}

\vspace{1cm}


\section{Introduction} \label{secintroduction}

Quasi-Monte Carlo rules $\frac{1}{N} \sum_{n=1}^N f(\bsx_n)$, $\bsx_1,\ldots, \bsx_N \in [0,1]^s$, are equal weight integration formulas used to approximate integrals over the unit cube $[0,1]^s$, where $s$ is typically large. One can roughly divide quasi-Monte Carlo rules into lattice rules, see e.g. \cite{N92,SJ94}, and digital nets, see e.g. \cite{DP09,N92}. In this paper we focus on digital nets, the construction of which is based on linear algebra over finite fields, see \cite{DP09,N92}. In particular, we are interested in a special case of digital nets, so-called polynomial lattice rules, which are constructed using polynomials over finite fields; polynomial lattice rules were introduced in \cite{N92a}, see also \cite{DKPS05,DP09,N92}.

Studying the approximation of integrals using quasi-Monte Carlo methods, one wants to have information on the resulting integration errors. However, depending on the integrand under consideration, estimates of integration errors might be very conservative or unknown; a possible remedy to this problem is randomization, which allows us to obtain statistical information on integration errors, \cite{O95}. Popular choices of randomization methods are digital shifts, see e.g. \cite{DP05,DP09}, and scrambling as introduced by Owen~\cite{O95}, see also \cite{DP09, HY00, matou, O97a, O97, Y99, YH02, YH05, YM99}. In this paper, we focus on scrambling. In particular, we are interested in the variance of the estimator
\begin{equation} \label{eqdefest}
 \hat{I}(f) = \frac{1}{b^m} \sum^{b^m-1}_{h=0} f(\bsy_h),
\end{equation}
where the points $\left\{ \bsy_h \right\}^{b^m-1}_{h=0}$ are obtained by applying the scrambling algorithm to a polynomial lattice rule. Notice that $\hat{I}(f)$ is an unbiased estimator of $\int_{[0,1]^s} f(\bsx) {\rm d} \bsx$, that is, $\mathbb{E}(\hat{I}(f)) = \int_{[0,1]^s} f(\bsx) {\rm d} \bsx$, see \cite{O95}.

The variance of the estimator given in Equation \eqref{eqdefest} admits the representation, see \cite{O97a}, 
\begin{equation} \label{eqintrovarest}
\Var(\hat{I}(f)) = \frac{1}{N} \sum_{\bsl \in \nn^s_0 \setminus  \left\{ \bs0 \right\}} \Gamma_{\bsl} \sigma^2_{\bsl}(f),
\end{equation}
where $N$ is the number of quadrature points. Equation \eqref{eqintrovarest} holds for any estimator obtained by applying the scrambling algorithm to a point set $\left\{ \bsx_{h} \right\}^{b^m-1}_{h=0}$ such that $\bsx_{h} \in [0,1)^s$. Here, the values $\Gamma_{\bsl}$ are the so-called gain coefficients which depend only on the quadrature points and the values $\sigma_{\bsl}(f)$ depend only on the integrand $f$. They are derived from the crossed and nested Anova decomposition of $f$, see \cite{O97a}, and can be expressed in terms of Haar coefficients of the function $f$, see \cite{O97a}, or also as a sum of certain Walsh coefficients of $f$, see \cite[Section~13.2]{DP09}.  In this sense, Equation \eqref{eqintrovarest} shows that $\Var(\hat{I}(f))$ can be expressed as a weighted sum of gain coefficients, where we interpret the $\sigma^2_{\bsl}(f)$ as weights.

In our investigations we consider a space of functions for which
$\sigma_{\bsl}(f)$ has a certain rate of decay. More precisely, for
$0 < \alpha \le 1$ we introduce a norm of the form
\begin{equation}\label{norm_unweighted}
\|f\|_\alpha = \sup_{\bsl \in \nn_0^s} b^{\alpha |\bsl|_1} \sigma_{\bsl}(f),
\end{equation}
where $|\bsl|_1 = l_1 + \cdots + l_s$ for $\bsl = (l_1,\ldots,
l_s)$. We show that $\alpha$ is related to the smoothness of $f$ in
the following sense:
\begin{quote}
If $f \in L_2([0,1]^s)$ has bounded variation of order $\alpha$,
then $\|f\|_\alpha < \infty$.
\end{quote}
See Corollary~\ref{corHoeldcontmodcont} for details.

From \eqref{norm_unweighted} we obtain $\sigma_{\bsl}(f) \le
b^{-\alpha |\bsl|_1} \|f\|_\alpha$ and by substituting this formula
into \eqref{eqintrovarest} we obtain
\begin{equation}\label{introboundvar}
\Var(\hat{I}(f)) \le \frac{1}{N} \sum_{\bsl \in \nn_0^s \setminus \{\bs0\}} \Gamma_{\bsl} b^{-2 \alpha |\bsl|_1} \|f\|_\alpha^2.
\end{equation}
To construct polynomial lattice rules of high quality we use
\begin{equation}\label{sum_criterion}
\frac{1}{N} \sum_{\bsl \in \nn_0^s \setminus \{\bs0\}} \Gamma_{\bsl} b^{-2 \alpha |\bsl|_1}
\end{equation}
as quality criterion. Notice that the sum \eqref{sum_criterion} only depends on the quadrature points and not on the function $f$. We show that \eqref{sum_criterion} has a simple closed form for any $0 < \alpha \le 1$ which can easily be computed if the quadrature points are a digital net. The case $\alpha = 0$ needs to be excluded since in this case \eqref{sum_criterion} is infinite.

Our aim is to find polynomial lattice rules for which the weighted sum of gain coefficients \eqref{sum_criterion} is minimized.  It is known from \cite{O97a, O97,YH02} that a small quality parameter $t$ of a digital $(t,m,s)$-net yields small gain coefficients. In fact, one has $\Gamma_{\bsl} = 0$ for all $\bsl \in \nn_0^s \setminus \{\bs0\}$ such that $|\bsl|_1 \le m- t$. Here, on the other hand, we aim to minimize \eqref{sum_criterion} since it can be used to bound the variance of the estimator $\Var(\hat{I}(f))$. In other words, we minimize the upper bound on the variance \eqref{introboundvar} for all functions $f$ with $\|f\|_\alpha < \infty$ over the class of polynomial lattice rules.

We introduce additional parameters $\bsgamma = (\gamma_j)_{1 \le j \le s}$ in the norm \eqref{norm_unweighted} in the sense of \cite{SW98}. In this case the criterion \eqref{sum_criterion} depends on the additional parameters $\bsgamma = (\gamma_j)_{1 \le j \le s}$, in which case a small quality parameter $t$ does not necessarily yield the smallest possible gain coefficients anymore. In such a situation component-by-component constructions \cite{SR02} have proven useful since one can then optimize the quadrature points also with respect to the $\gamma_j$. This is also the approach taken here. More precisely, we show that by constructing polynomial lattice rules component-by-component one obtains a convergence of $$\Var(\hat{I}(f)) = O(N^{-(2\alpha+1) + \delta}) \quad \mbox{for any } \delta > 0.$$ Apart from $\delta$, which can be arbitrarily close to $0$, this rate is best possible as shown in Section~\ref{seclowerboundwcv} for a large class of randomized algorithms. Further, if $\sum_{j=1}^\infty \gamma_j < \infty$, then the bound \eqref{introboundvar} does not depend on the dimension~$s$. This result is stronger than what can be obtained for $(t,m,s)$-nets, since the increase of the $t$ value of the form $t \approx s$ prevents one from obtaining a bound independent of the dimension only assuming that $\sum_{j=1}^\infty \gamma_j < \infty$. This condition is also necessary, see \cite{YH05} for a result on a related space. Hence the rules we construct here are simultaneously optimal in terms of the convergence rate as well as in terms of their dependence on the dimension.

Notice that the additional parameters $\bsgamma = (\gamma_j)_{j\ge 1}$, as introduced in \cite{SW98}, have been found to be very useful from both a theoretical and a practical point of view. They allow us to associate more importance with some variables than with others, which ties in nicely with concepts such as effective dimension, see, e.g., \cite{CaflischMOVal}. This has been used to explain the success of quasi-Monte Carlo rules in finance.

We now give the structure of the paper. In Section \ref{secprelim}, we recall the definition of polynomial lattice rules, the scrambling algorithm and define the function space studied in this paper. The variance of estimators based on scrambled polynomial lattice rules is studied in Section \ref{secvarscrambledpolyrules}.  Next, in Sections \ref{seccbc} and \ref{seckorobov}, we construct polynomial lattice rules for which the variance of the associated estimator converges at a rate of $N^{-(1 + 2 \alpha) + \delta}$, for all $\delta >0$. The constructions are based on a component-by-component approach and the Korobov construction respectively. In Section \ref{seclowerboundwcv}, we define a large class of randomized algorithms, which includes adaptive ones, and consequently establish that the variance of any estimator based on an algorithm from this class converges at most a rate of $N^{-(1+2 \alpha)}$ for the function space under consideration in this paper. Hence our constructions  are almost optimal for the class of algorithms defined in Section~\ref{seclowerboundwcv}. In Section \ref{secimplementation}, we study the implementation of the component-by-component approach, in particular, we show how to reduce the computational effort associated with it. This implementation is made use of in Section \ref{secnumexp}, where we compare the performance of scrambled polynomial lattice rules constructed in Section \ref{seccbc} to the performance of scrambled digital nets. 

\section{Preliminaries} \label{secprelim}

In this section, we define polynomial lattice rules, recall the scrambling algorithm and introduce the function space under consideration in this paper.

\subsection{Polynomial Lattice Rules} \label{subsecpollattrules}

Polynomial lattice rules were introduced in \cite{N92a}, see also \cite{DKPS05, DP09, N92}. We fix a prime $b$ and denote by $\zz_b$ the finite field containing $b$ elements and by $\zz_b((x^{-1}))$ the field of formal Laurent series over $\zz_b$. Elements of $\zz_b((x^{-1}))$ are formal Laurent series,
\begin{displaymath}
 L=\sum^{\infty}_{l=w} t_l x^{-l} \, ,
\end{displaymath}
where $w$ is an arbitrary integer and all $t_l \in \zz_b$. The field $\zz_b ((x^{-1}))$ contains the field of rational functions over $\zz_b$ as a subfield. Finally, the set of polynomials over $\zz_b$ is denoted by $\zz_b[x]$. For an integer $m$, we denote by $v_m$ the map from $\zz_b((x^{-1}))$ to $[0,1)$ defined by
\begin{displaymath}
 v_m \left( \sum^{\infty}_{l=w} t_l x^{-l} \right) = \sum^{m}_{l=\max(1,w)} t_l b^{-l} \, .
\end{displaymath}
The following definition of polynomial lattice rules stems from \cite{N92a}, see also \cite{DP09, N92}.

\begin{definition} \label{defpolylatt} Let $b$ be prime and $m$ be an integer. For a given dimension $s \geq 1$, choose $p(x) \in \zz_b[x]$ with $deg(p(x))=m$ and $q_1(x), \dots, q_s(x) \in \zz_b[x]$. For $0 \leq h < b^m$ let $h=h_0 + h_1 b + \dots + h_{m-1} b^{m-1}$ be the $b$-adic expansion of $h$. With each such $h$ we associate the polynomial
\begin{displaymath}
 \overline{h}(x) = \sum^{m-1}_{r=0} h_r x^r \in \zz_{b}[x] \, .
\end{displaymath}
Then $S_{p,m}(\bsq)$, where $\bsq=(q_1,\dots,q_s)$, is the point set consisting of the $b^m$ points
\begin{displaymath}
\bsx_h = \left( v_m \left(\frac{ \overline{h}(x) q_1(x) }{p(x)} \right) , \dots, v_m \left( \frac{ \overline{h}(x) q_s(x) }{p(x)} \right) \right) \in [0,1)^s \, ,
\end{displaymath}
for $0 \leq h < b^m$. A quasi-Monte Carlo rule using the point set $S_{p,m}(\bsq)$ is called a polynomial lattice rule.
\end{definition}

We remark that polynomial lattice point sets are also digital nets, see \cite{DP09,N87,N92}.

For the remainder of the paper, we use the following notation: We write $\vec{h}$ for vectors over $\zz_b$ and $\bsh$ for vectors over $\zz$ or $\rr$. Polynomials over $\zz_b$ are denoted by $h(x)$ and vectors of polynomials by $\bsh(x)$. Furthermore, given an integer $h$ with $b$-adic expansion $h=\sum^{\infty}_{r=0} h_r b^r$, we denote the associated polynomial by
\begin{displaymath}
 \overline{h}(x) = \sum^{\infty}_{r=0} h_r x^r \, .
\end{displaymath}
For arbitrary $\bsh(x)=(h_1(x), \dots, h_s(x) ) \in \zz_b[x]^s$ and $\bsq(x)=(q_1(x),\dots,q_s(x)) \in \zz_b[x]^s$, we define the ``inner product''
\begin{displaymath}
 \bsh(x) \cdot \bsq(x) = \sum^s_{j=1} h_j q_j(x) \in \zz_b[x]
\end{displaymath}
and we write $q(x) \equiv 0 \pmod{ p(x) }$ if $p(x)$ divides $q(x)$ in $\zz_b[x]$.

Finally, we introduce the dual lattice which plays an important role in numerical integration, see \cite{DKPS05, DP09}, which requires us to introduce the following function: For a non-negative integer $k$ with $b$-adic expansion $k=k_0 + k_1 b + \dots $ we write ${\rm tr}_m(k) = k_0 + k_1 b + \dots + k_{m-1} b^{m-1} $ and thus the associated polynomial
\begin{displaymath}
 {\rm tr}_{m}(k)(x) = k_0 + k_1 x + \dots k_{m-1} x^{m-1} \in \zz_b[x]
\end{displaymath}
has degree $< m$. For a vector $\bsk \in \nn^s_0$, ${\rm tr}_{m} (\bsk)$ is defined componentwise.

\begin{definition} \label{defdualpollatt} Let $\bsq(x)=(q_1(x), \dots,q_s(x)) \in \zz_b[x]^s$, then the dual polynomial lattice of $S_{p,m}( \bsq)$ is given by
\begin{eqnarray*}
\Dcal & = & \Dcal_{p}(\bsq) \\  & = & \{ \bsk \in \nn^s_0: \\ && \;\;\; {\rm tr}_{m}(k_1)(x) + \dots + {\rm tr}_{m}(k_2)(x) q_2 + \dots + {\rm tr}_{m} (k_s)(x) q_s(x) \equiv \bs0 \pmod{ p(x) } \}.
\end{eqnarray*}
\end{definition}

\subsection{The Scrambling Algorithm} \label{subsecscramblingalg}

The scrambling algorithm is a randomization algorithm which was introduced by Owen, see \cite{O95} and also \cite{HHY04,HY00,O97a,O97,Y99,YH02,YM99}.

We now describe the scrambling algorithm using a generic point $\bsx \in [0,1)^s$, where $\bsx=(x_1,\dots,x_s)$ and
\begin{displaymath}
 x_j=\frac{\xi_{j,1}}{b} + \frac{\xi_{j,2}}{b^2} + \dots .
\end{displaymath}
Then the scrambled point shall be denoted by $\bsy \in [0,1)^s$, where $\bsy=(y_1,\dots,y_s)$,
\begin{displaymath}
 y_j=\frac{\eta_{j,1}}{b} + \frac{\eta_{j,2}}{b^2} + \dots .
\end{displaymath}
The permutation applied to $\xi_{j,l}$, $j=1,\dots,s$ depends on $\xi_{j,k}$, for $1 \leq k < l$. In particular, $\eta_{j,1}=\pi_j(\xi_{j,1})$, $\eta_{j,2}=\pi_{j,\xi_{j,1}}(\xi_{j,2})$, $\eta_{j,3}=\pi_{j,\xi_{j,1},\xi_{j,2}}(\xi_{j,3})$ and in general
\begin{displaymath}
 \eta_{j,k}=\pi_{j,\xi_{j,1},\dots,\xi_{j,k-1}}(\xi_{j,k}) \, , k \geq 2 \, ,
\end{displaymath}
where $\pi_j$ and $\pi_{j,\xi_{j,1},\dots,\xi_{j,k-1}}$, $k \geq 2$ are random permutations of $\left\{ 0,1,\dots,b-1\right\}$. We assume that permutations with different indices are mutually independent. Also, if we apply the scrambling algorithm to $\bsx$ to obtain $\bsy$, then $\bsy$ is uniformly distributed in $[0,1)^s$, see \cite[Proposition 2]{O95}. Finally, it was shown in \cite{O95} that the scrambling algorithm preserves the $(t,m,s)$-net property with probability $1$, i.e applying the scrambling algorithm to the points of a $(t,m,s)$-net results in a $(t,m,s)$-net with probability $1$.

\subsection{A Weighted Walsh Function Space based on Variance} \label{subsecweightfunspace}

In this section, we introduce the function space under consideration in this paper. In particular, we consider weighted spaces and for this purpose, we introduce a sequence of positive, non-increasing weights $\bsgamma=(\gamma_j)^{\infty}_{j=1}$. The purpose of the  weights is to model the importance of the different variables and we point out that the idea stems from \cite{SW98}. For $s \in \nn$, let $[s]=\left\{1,\dots,s \right\}$ and for $\uu \subseteq [s]$ let $ \bsgamma_{\uu} := \prod_{j \in \uu} \gamma_j$ be the weight associated with the projection onto coordinates whose index is contained in $\uu$.

Walsh functions have been an important tool in the analysis of digital nets; in \cite{LT94}, Walsh functions were used for the first time to analyze nets and the connection between numerical integration using digital nets and Walsh functions was made in \cite{DP05}, see also \cite{DP09}.

We now briefly recall the definition of Walsh functions. Let $\nn_0$ denote the set of nonnegative and $\nn$ the set of positive integers. Each $k \in \nn$ has a unique $b$-adic representation $k=\sum^a_{i=0} k_i b^i$ with digits $k_i \in \left\{0, \dots,b-1 \right\}$ for $0 \leq i \leq a$, where $k_a \neq 0$. For $k=0$ we have $a=0$ and $k_0=0$. Similarly, each $x \in [0,1)$ has a $b$-adic representation $x=\sum^{\infty}_{i=1} \xi_i b^{-i}$ with digits $\xi_i \in \left\{ 0,\dots,b-1\right\}$ for $i \geq 1$. This representation is unique in the sense that infinitely many of the $\xi_i$ must differ from $b-1$. We define the $k$th Walsh function in base $b$, $\wal_k :[0,1) \rightarrow \cc $ by
\begin{displaymath}
 \wal_{k}(x) := \exp( 2 \pi \icomp (\xi_1 k_0 + \dots + \xi_{a+1} k_a)/b ) \, .
\end{displaymath}
For dimension $s \geq 2$ and vectors $\bsk =(k_1,\dots,k_s) \in \nn^s_0$ and $\bsx=(x_1,\dots,x_s) \in [0,1)^s$ we define $\wal_{\bsk} : [0,1)^s \rightarrow \cc$ by
\begin{displaymath}
 \wal_{\bsk} (\bsx) := \prod^s_{j=1} \wal_{k_j}(x_j) \, .
\end{displaymath}
Studying integration errors resulting from the approximation of an integral based on a digital net or a polynomial lattice rule, it is useful to consider the Walsh series of the integrand $f$. For $f \in L_2([0,1]^s)$, the Walsh series of $f$ is given by
\begin{equation} \label{eqdefWalshseriesf}
 f(\bsx) \sim \sum_{\bsk \in \nn^s_0} \hat{f}(\bsk) \wal_{\bsk}(\bsx) \, ,
\end{equation}
where the Walsh coefficients $\hat{f}(\bsk)$ are given by
\begin{displaymath}
 \hat{f}(\bsx)=\int_{[0,1]^s} f(\bsx) \overline{\wal_{\bsk}(\bsx)} d \bsx \, .
\end{displaymath}
We do not necessarily have equality in Equation \eqref{eqdefWalshseriesf}, however, the completeness of the Walsh function system $\left\{ \wal_{\bsk} : \bsk \in \nn^s_0 \right\}$ (see for instance \cite[Appendix A]{DP09}), implies that
\begin{displaymath}
 \Var[f] = \sum_{\bsk \in \nn^s_0 \setminus \left\{ \bs0 \right\} } \vert \hat{f}(\bsk) \vert^2 \, ,
\end{displaymath}
where $\Var[f] = \int_{[0,1]^s} (f(\bsx) - \overline{f} )^2 d \bsx$, and where $\overline{f} = \int_{[0,1]^s} f(\bsx) d \bsx$.

Let $\sigma^2_{(\bsl_{\uu}, \bs0)} = \sum_{\bsk \in L_{(\bsl_{\uu}, \bs0)}} \vert \hat{f}(\bsk) \vert^2$, where
\begin{equation} \label{eqdefLlu}
L_{(\bsl_{\uu},\bs0)} = \left\{ \bsk \in \nn^s_0 : b^{l_i -1} \leq k_i < b^{l_i} \textrm{ for } i \in  \uu \textrm{ and } k_i = 0 \textrm{ for } i \in [s] \setminus \uu  \right\}.
\end{equation}
Further let $\vert \bsl \vert_{1} = \sum^s_{j=1} l_j$ for $\bsl=(l_1,\dots,l_s)$. For $0 < \alpha \le 1$ we define a weighted norm for functions $f \in L_2([0,1]^s)$ by
\begin{equation}\label{defnormalpha}
\|f\|_\alpha = \max_{\uu \subseteq [s]} \gamma_\uu^{-1/2}
\sup_{\bsl_\uu \in \mathbb{N}^{|\uu|}} b^{\alpha |\bsl_\uu|_1}
\sigma_{(\bsl_\uu, \bs0)}(f).
\end{equation}
For $0 < \alpha \le 1$ define a space $V_{\alpha,s,\bsgamma} \subseteq L_2 ([0,1]^s)$ consisting of all functions $f$ for which $\|f\|_\alpha < \infty$. (One could of course use some $\ell_p$ norm instead of the supremum-norm to define $\|\cdot \|_\alpha$ and the function space, but these do not yield a quality criterion of the form \eqref{sum_criterion} which can be used for our construction, see \eqref{eqbound} and Lemma~\ref{theocompbound} below.)

The following observation stems from \cite[Section 13.5]{DP09}: For
a subinterval $J = \prod_{i=1}^s [x_i, y_i)$ with $0 \le x_i < y_i
\le 1$ and a function $f:[0,1)^s \rightarrow \real$, let the
function $\Delta(f,J)$ denote the alternating sum of $f$ at the
vertices of $J$ where adjacent vertices have opposite signs. (Hence
for $f = \prod_{i=1}^s f_i$ we have $\Delta(f,J) = \prod_{i=1}^s
(f_i(x_i) - f_i(y_i))$.)

We define the generalized variation in the sense of
Vitali\index{variation!fractional order} of order $0 < \alpha \le 1$
by
\begin{equation*}
V^{(s)}_{\alpha}(f) = \sup_{{\Pcal}} \left(\sum_{J \in \Pcal}
\Vol(J)
\left|\frac{\Delta(f,J)}{\Vol(J)^{\alpha}}\right|^{2}\right)^{1/2},
\end{equation*}
where the supremum is extended over all partitions $\Pcal$ of
$[0,1]^s$ into subintervals and $\Vol(J)$ denotes the volume of the
subinterval $J$.

For $\alpha = 1$ and if the partial derivatives of $f$ are
continuous on $[0,1]^s$ we also have the formula
\begin{equation*}
V_{1}^{(s)}(f) = \left(\int_{[0,1]^s} \left|\frac{\partial^s
f}{\partial x_1\cdots \partial x_s} \right|^{2} \rd
\bsx\right)^{1/2}.
\end{equation*}

Until now we did not take projections to lower-dimensional faces
into account.

For $\emptyset \neq u \subseteq\cS$, let $V_{\alpha}^{(|u|)}(f_u;u)$
be the generalized Vitali variation with coefficient $0 < \alpha \le
1$ of the $|u|$-dimensional function
$$f_u(\bsx_u) = \int_{[0,1)^{s-|u|}} f(\bsx) \rd \bsx_{\cS\setminus
u}.$$ For $u = \emptyset$ we have $f_\emptyset = \int_{[0,1)^{s}}
f(\bsx) \rd \bsx_{\cS}$ and we define
$V_{\alpha}^{(|\emptyset|)}(f_\emptyset;\emptyset) = |f_\emptyset|$.
Then
\begin{equation}\label{eq_varhk}
V_{\alpha}(f) = \left(\sum_{u \subseteq\cS}
\left(V^{(|u|)}_{\alpha}(f_u;u) \right)^{2}\right)^{1/2}
\end{equation}
is called the generalized Hardy and Krause variation of $f$ on
$[0,1]^s$.

A function $f$ for which $V_{\alpha}(f) < \infty$ is said to be of
finite variation of order $\alpha$.

The following result is from \cite{DG10} and
\cite[Section~13.5]{DP09}.

\begin{corollary} \label{corHoeldcontmodcont}
Let $b \ge 2$ be a natural number and let $f \in L_2([0,1]^s)$ have
bounded variation $V_\alpha(f) < \infty$ of order $0< \alpha \leq
1$. Then
\begin{displaymath}
\|f\|_\alpha \le \max\left(\|f\|_{L_2} \gamma_{\emptyset}^{-1},
V_\alpha(f)  \max_{\emptyset \neq \uu \subseteq [s]}
\gamma_\uu^{-1/2} (b-1)^{(\alpha - 1/2)_+ \cardu}\right).
\end{displaymath}
\end{corollary}

Hence every function $f \in L_2([0,1]^s)$ which has bounded
variation of order $0 < \alpha \leq 1$ is in
$V_{\alpha,s,\bsgamma}$. For the extreme case $\alpha = 0$ one
obtains $V_{0,s,\bsgamma} = L_2([0,1]^s)$, but this case is not
included in our investigations since the criterion
\eqref{sum_criterion} is infinite in this case, see \eqref{eqbound}
and Lemma~\ref{theocompbound} below.

\section{The Variance of Estimators based on Scrambled Polynomial Lattice Rules} \label{secvarscrambledpolyrules}

In this section, we discuss the variance of the estimator
\begin{equation} \label{eqdefestimator}
 \hat{I}(f) = \frac{1}{b^m} \sum^{b^m-1}_{h=0} f(\bsy_{h}) \, ,
\end{equation}
where the points $\bsy_{0}, \dots, \bsy_{b^m-1}$ are obtained by
applying the scrambling algorithm to a digital $(t,m,s)$-net over
$\zz_b$.

We use the following notation: For a non-negative integer $k$ with
$b$-adic expansion
\begin{displaymath}
 k=k_0 + k_1 b + \dots ,
\end{displaymath}
we write $\vec{k}=(k_0,k_1,\dots)^{\top}$, which is an
infinite-dimensional vector, and we use
\begin{displaymath}
 {\rm tr}_{m} (\vec{k}) = (k_0,k_1,\dots,k_{m-1})^{\top} \, .
\end{displaymath}

We now introduce the integration problem studied in this paper, in
particular, we are interested in the worst-case variance of
multivariate integration in $V_{\alpha,s,\bsgamma}$ using a
scrambled quasi-Monte Carlo rule $Q_{b^m,s}$:
\begin{displaymath}
\Var(Q_{b^m,s}, V_{\alpha,s,\bsgamma}) = \sup_{f \in V_{\alpha,s,\bsgamma}, \|f\|_\alpha \le 1} \Var[ \hat{I}(f, Q_{b^m,s}) ] \, ,
\end{displaymath}
where $\hat{I}(f,Q_{b^m,s})$ denotes the estimator based on the
point set obtained by applying the scrambling algorithm to
$Q_{b^m,s}$. We denote the quasi-Monte Carlo rule based on a
polynomial lattice rule $S_{p,m}(\bsq)$ by $Q_{b^m,s}(\bsq)$ and the
associated worst-case variance by $\Var(Q_{b^m,s}(\bsq),
V_{\alpha,s,\bsgamma} )$. For $k = \kappa_0 + \kappa_1 b + \cdots +
\kappa_{a-1} b^{a-1} \in \mathbb{N}_0$ let
\begin{equation*}
r_{\alpha, \gamma}(k) = \left\{\begin{array}{ll} 1 & \mbox{if } k = 0, \\
\gamma \frac{b}{(b-1) b^{\alpha a}} & \mbox{if } k > 0.
\end{array} \right.
\end{equation*}
For $\bsk = (k_1, \ldots, k_s) \in \mathbb{N}_0^s$ let
$r_{\alpha,\bsgamma}(\bsk) = \prod_{j=1}^s
r_{\alpha,\gamma_j}(k_j)$.

The next corollary gives a bound on the quantity
$\Var(Q_{b^m,s}(\bsq), V_{\alpha,s,\bsgamma} )$.

\begin{corollary} \label{corrboundworstcasevar}
Let $0 < \alpha \le 1$, $\bsq \in \zz_b[x]^s$ be a generating vector
for a polynomial lattice rule with modulus $p$, and
$\Var(Q_{b^m,s}(\bsq), V_{\alpha,s,\bsgamma} )$ be defined as above.
Then
\begin{displaymath}
\Var(Q_{b^m,s}(\bsq), V_{\alpha,s,\bsgamma} ) \leq \sum_{\bsk \in
\Dcal_{p}(\bsq) \setminus \{\bs0\}} r_{2 \alpha + 1,
\bsgamma}(\bsk),
\end{displaymath}
where $ \Dcal_p (\bsq)$ is the dual polynomial lattice.
\end{corollary}

\begin{proof}
The corollary follows from the following facts: For any $f \in
L_2([0,1]^s)$, let $\hat{I}(f)$ be given by Equation
\eqref{eqdefestimator} and $\left\{ \bsx_0, \dots, \bsx_{b^m-1}
\right\}$ be a digital $(t,m,s)$-net over $\zz_b$ with generating
matrices $C_1,\dots,C_s$ over $\zz_b$, then we have (see \cite{YH02} or also
\cite[Section~13.5]{DP09})
\begin{displaymath}
\Var[ \hat{I} (f) ] = \sum_{\emptyset \neq \uu \subseteq [s]}
\frac{b^{\cardu}}{(b-1)^{\cardu}} \sum_{\bsl_{\uu} \in \nn^{\cardu}}
\frac{ \sigma^2_{(\bsl_{\uu},\bs0)} (f) }{b^{\vert \bsl_{\uu}
\vert_1}} \vert L_{(\bsl_{\uu},\bs0)} \cap \Dcal(C_1,\dots,C_s)
\vert,
\end{displaymath}
where $\Dcal(C_1,\dots,C_s) = \left\{ \bsk \in \nn^s_{0}: C^{\top}_1
{\rm tr}_m(\vec{k}_1) + \dots + C^{\top}_s {\rm tr}_m(\vec{k}_s) = \vec{0}
\right\}$ and $L_{(\bsl_{\uu}, \bs0)}$ is given in Equation
\eqref{eqdefLlu}. Furthermore, if the $C_1,\dots,C_s$ are the
generating matrices of the point set $S_{p,m}(\bsq)$, then for any
$\bsk \in \nn^s_{0} \setminus \left\{ \bs0 \right\}$ we have
\begin{displaymath}
 C^{\top}_1 {\rm tr}_{m}(\vec{k}_1) + \dots + C^{\top}_{s} {\rm tr}_{m}(\vec{k}_s) = \vec{0} \Leftrightarrow {\rm tr}_m(\bsk) \cdot \bsq \equiv 0 \pmod p \, ,
\end{displaymath}
which was first established in \cite[Lemma 4.40]{N92}. Using
\eqref{defnormalpha} and
\begin{equation*}
\sum_{\emptyset \neq \uu \subseteq [s]}
\frac{b^{\cardu}}{(b-1)^{\cardu}} \sum_{\bsl_{\uu} \in \nn^{\cardu}}
\frac{1}{b^{(2\alpha + 1)\vert \bsl_{\uu} \vert_1}} \vert
L_{(\bsl_{\uu},\bs0)} \cap \Dcal_p(\bsq) \vert = \sum_{\bsk \in
\Dcal_{p}(\bsq) \setminus \{\bs0\}} r_{2\alpha + 1, \bsgamma}(\bsk) \, ,
\end{equation*}
the result follows.
\end{proof}

We denote the bound in Corollary \ref{corrboundworstcasevar} by
\begin{equation} \label{eqbound}
B(\bsq, \alpha,\bsgamma) := \sum_{\bsk \in \Dcal_{p}(\bsq) \setminus
\{\bs0\}} r_{2\alpha+1, \bsgamma}(\bsk).
\end{equation}
This bound is almost the same as the square worst case error for
integration in a certain Walsh space considered in \cite{DKPS05},
see in particular \cite[Lemma~4.1]{DKPS05}.

As in \cite{DKPS05}, $B(\bsq,\alpha, \bsgamma)$ can easily be
computed and therefore be used as a quality criterion for polynomial
lattice rules. We write $\log_b$ for the logarithm in base $b$ and
we set $b^{2\alpha \lfloor \log_b 0 \rfloor} = 0$.

\begin{lemma} \label{theocompbound}
Let $B(\bsq, \alpha,\bsgamma)$ be given by Equation \eqref{eqbound}.
Then
\begin{equation} \label{eqformulabound}
 B(\bsq, \alpha,\bsgamma) = - 1+ \frac{1}{b^m} \sum^{b^m-1}_{h=0} \prod^s_{j=1} \left( 1 + \frac{b}{b-1} \gamma_j \phi_\alpha(x_{h,j}) \right),
\end{equation}
where for $x \in [0,1)$ we set
\begin{equation*}
\phi_\alpha(x) = \frac{b-1 - b^{2\alpha \lfloor \log_b x \rfloor
}(b^{2\alpha+1}-1)}{b(b^{2\alpha}-1)}.
\end{equation*}
\end{lemma}

A detailed proof of this result can be found in \cite{B}.

In the next remark, we show that if we construct a polynomial
lattice rule which achieves optimal convergence rates for functions
in $V_{\alpha,s, \bsgamma}$ for some given $0 < \alpha \le 1$, then
this polynomial lattice rule also achieves optimal convergence rates
for functions in $V_{\alpha',s,\bsgamma'}$ where $\alpha \le \alpha'
\le 1$. This means that the polynomial lattice rule constructed to
achieve optimal convergence rates for functions of smoothness
$\alpha$ adjusts itself to the optimal rate of convergence, as long
as the smoothness $\alpha'$ of the function under consideration
satisfies $\alpha' \geq \alpha$.

\begin{remark} \label{remselfadjust} Assume that for a fixed $\alpha$, $0 < \alpha \leq 1$, we have constructed a polynomial lattice rule $S_{p,m}(\bsq)$ such that
\begin{equation} \label{eqboundselfadjust}
B(\bsq, \alpha,\bsgamma) \leq C_{s, \alpha, \bsgamma,\delta} N^{-(1
+ 2 \alpha) + \delta},
\end{equation}
for all $\delta > 0$, where $C_{s, \alpha, \bsgamma,\delta}$ is permitted to depend on $s, \alpha, \bsgamma$ and $\delta$. We point out that explicit constructions of polynomial lattice rules satisfying Equation \eqref{eqboundselfadjust} are given in Sections \ref{seccbc} and \ref{seckorobov}. It follows immediately from Jensen's inequality, that
\begin{displaymath}
 B(\bsq, \alpha,\bsgamma)^{ \frac{1 + 2 \alpha'}{1 + 2 \alpha} } \geq B(\bsq, \alpha', \bsgamma^{\frac{1+ 2 \alpha'}{1+2 \alpha}}),
\end{displaymath}
for $\alpha \leq \alpha' \leq 1$. Making use of
Assumption~\eqref{eqboundselfadjust}, we conclude that
\begin{displaymath}
 B(\bsq, \alpha', \bsgamma^{\frac{1+ 2 \alpha'}{1 + 2 \alpha}}) \leq C^{\frac{1+2\alpha'}{1+2\alpha}}_{s, \alpha, \bsgamma,\delta} \; N^{-(1 + 2 \alpha') + \delta \frac{1+2\alpha'}{1+2\alpha}},
\end{displaymath}
 for all $\delta >0$. In particular, this observation motivates the construction of
polynomial lattice rules for which $\alpha < 1$, as the resulting
point sets still achieve optimal convergence rates for functions of
bounded variation of order $\alpha'$, where $\alpha \leq \alpha'
\leq 1$.
\end{remark}

\section{Component-By-Component Construction of Polynomial Lattice Rules} \label{seccbc}

In this section, we show how to construct a polynomial lattice rule
using a component-by-component approach so that the bound given in
Equation \eqref{eqbound} converges at a rate of $N^{-1 - 2 \alpha +
\delta}$, for any $\delta > 0$. We remark that in \cite[Theorem
13.24]{DP09}, the corresponding result for digital nets was
presented. A component-by-component (CBC) approach was first
considered in \cite{SR02} in the context of constructing lattice
rules. Subsequently, the CBC algorithm has been applied to the
construction of polynomial lattice rules in \cite{DKPS05}.

We use $R_{b,m}$ to denote the set of all non-zero polynomials in $\zz_b[x]$ with degree at most $m-1$, i.e.
\begin{displaymath}
 R_{b,m} := \left\{ q \in \zz_b[x] : deg(q) < m \textrm{ and } q \neq 0 \right\} \, .
\end{displaymath}
It is clear that $\vert R_{b,m} \vert = b^m-1$ and furthermore it follows from the construction principle described in Subsection \ref{subsecpollattrules} that the polynomials $q_j$ can be restricted to $R_{b,m}$. Algorithm \ref{algcbc1} gives the CBC algorithm.

\begin{algorithm}[h!] \label{algcbc}
\caption{CBC algorithm}
\begin{algorithmic}[1]\label{algcbc1}
\REQUIRE $b$ a prime, $s,m \in \nn$ and weights $\bsgamma =
(\gamma_j)_{j \geq 1}$. \STATE Choose an irreducible polynomial $p
\in \zz_b[x]$, with $\deg(p)=m$. \STATE Set $q_1=1$. \FOR{$d=2$ to
$s$} \STATE find $q_d \in R_{b,m}$ by minimizing $B((q_1,\dots,q_d),
\alpha,\bsgamma)$ as a function of $q_d$. \ENDFOR \RETURN
$\bsq=(q_1,\dots,q_s)$.
\end{algorithmic}
\end{algorithm}

The next theorem shows that Algorithm \ref{algcbc1} indeed
constructs a $\bsq^*_d \in R^d_{b,m}$ so that
$B((q^*_1,\dots,q^*_d),\alpha,\bsgamma)$ converges at a rate of
$N^{-1 - 2 \alpha + \delta}$, for any $\delta >0$.

\begin{theorem} \label{theocbc1} Let $b$ be prime and $p \in \zz_b[x]$ be irreducible, with $deg(p)=m \geq 1$. Suppose $(q^*_1,\dots,q^*_s) \in R^s_{b,m}$ is constructed using Algorithm \ref{algcbc1}. Then for all $d=1,\dots,s$ we have
\begin{displaymath}
B((q^*_1,\dots,q^*_d),\alpha,\bsgamma) \leq \frac{1}{(b^m-1)^{1/
\lambda}} \prod^d_{j=1} \left[ 1 + \gamma^{\lambda}_j
C_{b,\alpha,\lambda} \right]^{1 / \lambda},
\end{displaymath}
for all $\frac{1}{2 \alpha +1} < \lambda \leq 1$ where
\begin{equation}\label{calphablambda}
C_{b,\alpha,\lambda} =
\max\left(\frac{1}{(b^{2\alpha}-1)^{\lambda}},
\frac{(b-1)^{1-\lambda}}{b^{2 \alpha \lambda } - b^{1-\lambda}}
\right).
\end{equation}
\end{theorem}

A proof of this result can be obtained by making a few modifications
to the proof of \cite[Theorem~4.4]{DKPS05}, which is included in \cite{B}. The additional term
$(b^{2\alpha}-1)^{-\lambda}$ in the definition of
$C_{b,\alpha,\lambda}$ arises from the one-dimensional case for
which we have (we assume $k = \kappa_0 + \kappa_1 b + \cdots +
\kappa_{a-1} b^{a-1}$)
\begin{eqnarray*}
B((1), \alpha, \bsgamma) & = & \gamma_1 \frac{b}{b-1} \sum_{k=1, b^m
| k}^\infty b^{-\alpha a} \\ &  = & \gamma_1 \frac{b}{b-1}
\sum_{l=m+1}^\infty (b-1) b^{l-m-1} b^{-(2\alpha+1) l} \\ & = &
\frac{1}{b^{(2\alpha+1)m}} \frac{\gamma_1}{b^{2\alpha}-1} \\ & \le &
\frac{1}{(b^m-1)^{1/\lambda}} [1 + \gamma_1^\lambda
(b^{2\alpha}-1)^{-\lambda}]^{1/\lambda},
\end{eqnarray*}
for all $\frac{1}{2\alpha+1} < \lambda \le 1$. The induction with
respect to the dimension can be carried out as in the proof of
\cite[Theorem~4.4]{DKPS05}.

The next result discusses the tractability of Algorithm
\ref{algcbc1}.
\begin{corollary} \label{cortraccbc}
Let $b$ be prime, $p \in \zz_b[x]$ be irreducible with $deg(p)=m
\geq 1$ and $N=b^m$. Suppose $\bsq^*_s \in R^s_{b,m}$ is constructed
using Algorithm \ref{algcbc1}. Then we have the following:
\begin{enumerate}
 \item
\begin{displaymath}
B(\bsq^*_s, \alpha,\bsgamma) \leq c_{s, \alpha,\bsgamma, \delta} \;
(N-1)^{-(2 \alpha +1) + \delta} \, , \textrm{ for all } 0 < \delta
\leq 2 \alpha,
\end{displaymath}
where \begin{displaymath} c_{s,\alpha,\bsgamma,\delta} =
\prod^s_{j=1} \left[ 1 + \gamma^{\frac{1}{2 \alpha +1 - \delta}}_j
C_{b,\alpha,(2\alpha+1-\delta)^{-1}} \right]^{2 \alpha +1 - \delta}
\,.
\end{displaymath}
\item Assume
\begin{equation} \label{eqstrongtracassum}
\sum^{\infty}_{j=1} \gamma^{\frac{1}{2 \alpha +1 - \delta}}_j <
\infty.
\end{equation}

Then $c_{s, \alpha,\bsgamma, \delta} \leq c_{\infty,
\alpha,\bsgamma, \delta} < \infty$ and we have
\begin{displaymath}
B(\bsq^*_s , \alpha,\bsgamma) \leq c_{\infty, \alpha,\bsgamma,
\delta} (N-1)^{-(2 \alpha +1) + \delta} \, , \textrm{ for all } 0 <
\delta \leq 2 \alpha .
\end{displaymath}
Thus the bound $B(\bsq^*_s , \alpha,\bsgamma)$ is bounded
independently of the dimension.
\item
Under the assumption
\begin{displaymath}
A:= \lim \sup_{s \rightarrow \infty} \frac{\sum^s_{j=1} \gamma_j}{s}
< \infty,
\end{displaymath}
we obtain $ c_{s, \alpha,\bsgamma, 2 \alpha} \leq \tilde{c}_{\eta}
(b-1)^{2 \alpha} s^{\frac{A + \eta}{b^{2 \alpha} -1} }$ and
therefore
\begin{displaymath}
B(\bsq^*_s, \alpha,\bsgamma) \leq \tilde{c}_{\eta} s^{\frac{A +
\eta}{b^{2 \alpha} -1} } (N-1)^{-1}
\end{displaymath}
for all $\eta >0 $, where the constant $\tilde{c}_{\eta}$ only
depends on $\eta$. Thus the bound $B(\bsq^*_s , \alpha,\bsgamma)$
satisfies a bound which depends only polynomially on the dimension.
\end{enumerate}
\end{corollary}

The proof is similar to the proof of \cite[Corollary 4.5]{DKPS05} and can be found in \cite{B}.

\section{Construction of Korobov Polynomial Lattice Rules} \label{seckorobov}

In this section, we construct Korobov polynomial lattice rules. The
ideas underlying this algorithm stem from the construction of
lattice rules, see \cite{K60}. We remark that the construction of
Korobov polynomial lattice rules has been examined in \cite{DKPS05},
see also \cite{LLNS96}. We denote the generating vector for the
Korobov polynomial lattice rule by $\psi(q)=(1,q,\dots,q^{s-1})
\pmod p$. As in Section \ref{seccbc}, we work with the bound
$B(\psi(q), \alpha, \bsgamma)$ and now state the algorithm showing
how to construct Korobov polynomial lattice rules.

\begin{algorithm}[h!] \label{algcbck}
\caption{Korobov algorithm}
\begin{algorithmic}[1]\label{algcbck1}
\REQUIRE $b$ a prime, $s,m \in \nn$ and weights $\bsgamma =
(\gamma_j)_{j \geq 1}$. \STATE Choose an irreducible polynomial $p
\in \zz_b[x]$, with $\deg(p)=m$. \STATE Find $q^* \in R_{b,m}$ by
minimizing $B(\psi(q), \alpha, \bsgamma)$.
\end{algorithmic}
\end{algorithm}

We obtain the following bound for $B(\psi(q^*), \alpha,\bsgamma)$,
where $q^*$ is constructed using Algorithm \ref{algcbck1}. The proof
of the following theorem can be obtained by making a few modifications
to the proof of \cite[Theorem~4.7]{DKPS05}, which are presented in \cite{B}.

\begin{theorem} \label{theokorbound} Let $b$ be prime, $s \geq 2$ and let $p \in \zz_b[x]$ be irreducible with $deg(p)=m \geq 1$. A minimizer $q^*$ obtained from Algorithm \ref{algcbck1} satisfies
\begin{displaymath}
 B( \psi(q^*), \alpha, \bsgamma) \leq  \frac{s^{1 / \lambda} }{(b^{m}-1)^{1 / \lambda}} \prod^s_{j=1} \left( 1 + \gamma^{\lambda}_j C_{b,\alpha,\lambda} \right)^{1 / \lambda},
\end{displaymath}
for all $\frac{1}{2 \alpha} < \lambda \leq 1$, where
$C_{b,\alpha,\lambda} > 0$ is given by \eqref{calphablambda}.
\end{theorem}

We point out that the bounds in Theorems \ref{theocbc1} and \ref{theokorbound} only differ by the additional factor $s^{1 / \lambda}$. We remark that the same observation was made in \cite{DKPS05} and is also known from the lattice rule case. This leads to the conclusion that the Korobov construction is inferior to the component-by-component construction.

In the next corollary, we discuss the tractability of Algorithm \ref{algcbck1}.

\begin{corollary} \label{cortrackorobov} Let $b$ be prime, $s \geq 2$, $p \in \zz_b[x]$ be irreducible with $deg(p)=m \geq 1$ and $N=b^m$. Suppose $q^* \in R_{b,m}$ is constructed using Algorithm \ref{algcbck1}. Then we have the following:
\begin{enumerate}
 \item \begin{displaymath}
B(\psi(q^*), \alpha,\bsgamma) \leq c_{s,\alpha,\bsgamma,s} s^{2
\alpha +1 - \delta} (N-1)^{-(2 \alpha +1) + \delta} \, \textrm{, for
all } 0 < \delta \leq 2 \alpha,
       \end{displaymath}
where
\begin{displaymath}
 c_{s, \alpha,\bsgamma,\delta} = \prod^s_{j=1} \left( 1 + \gamma_j^{\frac{1}{2 \alpha +1 - \delta}} C_{b,\alpha,(2\alpha+1-\delta)^{-1}} \right)^{2 \alpha +1 - \delta}.
\end{displaymath}
 \item Under the assumption
\begin{displaymath}
 A:= \lim sup_{s \rightarrow \infty} \frac{\sum^s_{j=1} \gamma_j }{\log s} < \infty
\end{displaymath}
we obtain
\begin{displaymath}
 c_{s, \alpha,\bsgamma, 2 \alpha} \leq \tilde{c}_{\eta} s^{ \frac{A + \eta}{b^{2 \alpha} -1} }
\end{displaymath}
and therefore
\begin{displaymath}
 B( \psi(q^*), \alpha, \bsgamma) \leq \tilde{c}_{\eta} s^{1 + \frac{A + \eta}{b^{2 \alpha} -1}} (N-1)^{-1},
\end{displaymath}
for all $\eta>0$, where the constant $\tilde{c}_{\eta}$ only depends
on $\eta$. Thus the bound $B( \psi(q^*), \alpha,\bsgamma)$ satisfies
a bound which depends only polynomially on the dimension.
\end{enumerate}
\end{corollary}

The proof is again similar to the proof of
Corollary~\cite[Corollary~4.8]{DKPS05} and can be found in \cite{B}.

\section{A Lower Bound on the Worst-Case Variance} \label{seclowerboundwcv}

In this section, we produce a lower bound on the worst-case variance
discussed in Section~\ref{secvarscrambledpolyrules}. As we rely on
\cite[Section 2.2.4, Proposition 1]{N88} to establish the result,
the class of algorithms to which our result applies is the same as
the class considered there. We now recall the definition of this class. Following \cite[Section
1.1]{N88}, we use the notation
\begin{displaymath}
 S(f) = \int_{[0,1]^s} f(\bsx) d \bsx,
\end{displaymath}
for $f \in V_{\alpha, s, \bsgamma}$ and consider approximating $S :
V_{\alpha,s,\bsgamma} \rightarrow \rr$ using a mapping $\tilde{S} :
V_{\alpha,s,\bsgamma} \rightarrow \rr$. As in \cite[Section
1.1]{N88}, we assume that in general, the function $f \in
V_{\alpha,s,\bsgamma}$ is not known, but we have some information on
$f$ available, which is denoted by $L$, where $L:
V_{\alpha,s,\bsgamma} \rightarrow H$ and an approximation $\tilde{S}
: V_{\alpha,s,\bsgamma} \rightarrow \rr$ only uses the information
$L$ if it can be written as follows $\tilde{S} = \varphi \circ L$,
where $\varphi : H \rightarrow \rr$ is a an arbitrary mapping, referred
to as an (idealized) algorithm in \cite{N88}. In particular, we
allow our approximation nodes to be chosen adaptively and define the
following information operator:
\begin{eqnarray*}
I_N & = & \left\{ L : V_{\alpha,s,\bsgamma} \rightarrow \rr^N
\vert L(f) = (f(\bsa_1), f(\bsa_2[f(\bsa_1)]), \dots, f(\bsa_N [ f(\bsa_1), \dots,
f(\bsa_{N-1}) ] ) ) \right. ,
\\ && \left. \textrm{ where } \bsa_1 \in [0,1]^s \mbox{ and } \bsa_i : \rr^{i-1} \rightarrow [0,1]^s \mbox{ for } i = 2, \ldots, s \right\}
\end{eqnarray*}
and we can now introduce the class of all approximations considered
in this section:
\begin{displaymath}
A_{N} = \left\{ \tilde{S} : V_{\alpha,s,\bsgamma} \rightarrow \rr \vert \tilde{S} = \varphi \circ L \textrm{  with } \varphi : \rr^N \rightarrow \rr \textrm{ and } L \in I_N \right\}.
\end{displaymath}
We remark that non-adaptive algorithms are of course included in
$A_N$, consider $\tilde{S} = \varphi \circ \overline{L}$, where
$\overline{L}(f) = ( f(\bsa_1), \dots f(\bsa_N)  )$. Now, following
\cite[Section 2.1]{N88}, we can define the randomized algorithms
considered in this paper, referred to as generalized Monte Carlo
methods in \cite{N88}: A random variable $Q=( Q ( \omega))_{ \omega
\in \Omega}$ is called a randomized algorithm in $A_N$ if $(\Omega, B
, \mu)$ is a probability space and $Q(\omega) \in A_N$ for all $\omega
\in \Omega$. The set of all randomized algorithms is denoted by
$^*C(A_N)$, hence randomly scrambled nets (and therefore polynomial lattice rules) are also included in this set. We now present the lower bound on the worst-case
variance, which applies to all randomized algorithms in $^*C(A_N)$.

\begin{theorem} \label{theolowboundvar}
Let $^*C(A_N)$, $V_{\alpha,s,\bsgamma}$ be defined as above. Then
\begin{displaymath}
\inf_{Q \in ^*C(A_N)} \sup_{f \in V_{\alpha,s,\bsgamma}} \Var(\hat{I}(f,Q)) \geq \tilde{C} N^{- 2\alpha - 1},
\end{displaymath}
for some constant $\tilde{C}$ independent of $N$ where $$\Var(\hat{I}(f,Q)) = \int_{\Omega} \left[\hat{I}(f,Q(\omega)) - \int_\Omega \hat{I}(f,Q(\omega')) {\rm d} \mu(\omega') \right]^2 {\rm d} \mu(\omega).$$
\end{theorem}

\begin{proof}
We remark that this proof follows along the lines of the proof of \cite[Theorem 10]{HHY04}. We only consider $s=1$, since integration in $V_{\alpha,1,\gamma_1}$ is no harder than integration in $V_{\alpha,s,\bsgamma}$ with $s>1$, as the one-dimensional space $V_{\alpha,1,\gamma_1}$ can be identified with the subspace of $V_{\alpha,s,\bsgamma}$ consisting of functions depending only on the first variable. We let $N$ be any given natural number and choose an integer $m$ such that
\begin{displaymath}
 b^{m-1} < 2 N \leq b^m \, .
\end{displaymath}
We define basic intervals
\begin{displaymath}
 B_{m,a} = \left[ \frac{a}{b^m} , \frac{a+1}{b^m} \right) \, , \, a=0,1,\dots,b^m-1 \, ,
\end{displaymath}
and let $g_a(x) = \I_{B_{m, a}}(x)$ be the characteristic function of
$B_{m, a}$. Then
\begin{displaymath}
 \int_{[0,1]} g_{a} (x) g_c(x) dx = \left\{\begin{array}{cc} b^{-m} & \mbox{if } a = c, \\ 0 & \mbox{ otherwise}. \end{array} \right.
\end{displaymath}
We now define
\begin{displaymath}
 g = \sum^{b^m-1}_{a=0} \xi_a g_a \, ,
\end{displaymath}
where $\xi_a \in \left\{1, -1 \right\}$ and bound
$\sigma^2_{l}(g)$. Using Plancharel's identity we obtain that for any $l \ge 0$ we have
\begin{equation*}
\sigma_l^2(g) \le \sum_{l'=0}^\infty \sigma_{l'}^2(g) = \int_0^1 g^2(x) d x = \sum_{a,c = 0}^{b^m-1} \xi_a \xi_c \int_0^1 g_a(x) g_c(x) d x = \frac{1}{b^m} \sum_{a=0}^{b^m-1} \xi_a^2 = 1.
\end{equation*}

Further, for $k \ge b^m$ we have
\begin{equation*}
\widehat{g}(k) = \int_0^1  g(x) \overline{\wal_k(x)} d x = \sum_{a=0}^{b^m-1} \xi_a \int_0^1  g_a(x) \overline{\wal_k(x)} d x = \sum_{a=0}^{b^m-1} \xi_a \int_{a/b^m}^{(a+1)/b^m} \overline{\wal_k(x)} d x = 0,
\end{equation*}
since $\int_{a/b^m}^{(a+1)/b^m} \overline{\wal_k(x)} dx = 0$ for $k \ge b^m$ and hence for $l > m$ we have
\begin{equation*}
\sigma_l^2(g) = \sum_{k=b^{l-1}}^{b^l-1} |\widehat{g}(k)|^2 =  0.
\end{equation*}

We set $f_a = \gamma_1 b^{-\alpha m} g_a$ for $a=0,1,\dots,b^m-1$. These $f_a$ have
disjoint support and $$\int_{[0,1]} f_a(x) dx \geq \gamma_1 b^{-(
\alpha+1) m}.$$ Set $$f=  \gamma_1 b^{-\alpha m} g = \sum^{b^m-1}_{a=0} \xi_a f_a,$$ then we get
$\sigma^2_l (f) \leq \gamma^2_1 b^{-2 \alpha m}$ for $0 \le l \le m$ and $\sigma_l^2(f) = 0$ for $l > m$. Hence
\begin{equation*}
\|f\|_\alpha = \gamma_1^{-1} \sup_{l \in \mathbb{N}} b^{\alpha l} \sigma_l(f) \le \gamma_1^{-1} \sup_{1 \le l \le m} b^{\alpha l} \gamma_1 b^{-\alpha m} \le 1
\end{equation*}
and the result follows now from \cite[Section~2.2.4, Proposition 1(ii)]{N88}.
\end{proof}

\begin{remark} \label{remexplicitalg}
For a large class of randomized algorithms, including adaptive ones, we have shown that the worst-case variance in the Walsh function space $V_{\alpha,s,\bsgamma}$ behaves like $N^{-(1+2 \alpha)}$. In Sections~\ref{seccbc} and \ref{seckorobov} we presented two algorithms which achieve worst-case variances of order $N^{-(1 + 2 \alpha ) + \delta}$, for all $\delta >0$, and are hence almost optimal for the class of algorithms $^*C(A_N)$.
\end{remark}

\section{Implementation of the Component-By-Component Algorithm} \label{secimplementation}

In this section, we show how to implement the CBC algorithm from Section \ref{seccbc}. Our approach is based on \cite{NC06}, but we simplify the algorithm using ideas from \cite{D10}. Using ideas from \cite{NC05,NC06}, we obtain, for $d \geq 2$,
\begin{eqnarray*}
\lefteqn{ B(\bsq,\alpha,\bsgamma) } \\ & = & \frac{1}{b^m}
\sum^{b^m-1}_{h=0} \prod^d_{j=1} \left( 1 + \frac{b}{b-1} \gamma_j
\phi_\alpha(\bsx_{h,j}) \right) -1
\\ & = &  \frac{1}{b^m} \prod^d_{j=1} \left( 1 + \frac{b}{b-1} \gamma_j \phi_\alpha(\bsx_{0,j}) \right) -1 + \frac{1}{b^m} \sum^{b^m-1}_{h=1} \bsp_{d-1}(h) \left(1 + \frac{b}{b-1} \gamma_d \phi_\alpha(\bsx_{h,d}) \right),
\end{eqnarray*}
where $$\bsp_{d-1}(h) =  \prod^{d-1}_{j=1} \left( 1 + \frac{b}{b-1}
\gamma_j \phi_\alpha(\bsx_{h,j}) \right).$$

Let $\omega \left( \frac{ \overline{h} \overline{q}_d }{p} \right) =
\phi_\alpha(\bsx_{h,d})$, where $\overline{h}$ and $\overline{q}_d$
denote the polynomials associated with $h$ and $q_d$ and $p$ denotes
the polynomial $p=p(x) \in \zz_b[x]$. Following \cite{NC05}, we now
introduce the following matrix
\begin{equation}
 \Omega_p = \left[ \omega \left( \frac{ \overline{h} \overline{q} }{p} \right) \right]_{ \begin{subarray}{c} q=1,\dots,b^m-1 \\ h=1, \dots, b^m-1 \end{subarray} } \, ,
\end{equation}
i.e. rows are indexed by $q$ and columns by $h$.

Let $\bsp_{d-1}=(\bsp_{d-1}(1), \dots, \bsp_{d-1}(b^m-1) )^{\top}$.
Following \cite{NC05}, we have an update rule for $\bsp_d$ given by
\begin{displaymath}
 \bsp_d = \diag \left(\left( \I_{(b^m-1) \times (b^m-1) } + \frac{b}{b-1} \gamma_d \Omega_p \right)  v_{q_d}  \right) \bsp_{d-1},
\end{displaymath}
where $\diag(\bsx)$ denotes the diagonal matrix with the elements of
$\bsx$ on its diagonal and zero elsewhere and where we use $v_j$ to
denote a selection vector with $1$ in position $j$ and $0$
elsewhere.

We now use the notation $B_{d-1} =
(B((\bsq_{d-1},\overline{1}),\alpha,\bsgamma), \ldots,
B((\bsq_{d-1}, \overline{b^m-1}), \alpha,\bsgamma))^\top$. Then
\begin{eqnarray*} 
B_{d-1} & = & \left[-1 + \frac{1}{b^m}
\prod^d_{j=1} \left( 1+ \frac{b}{b-1} \gamma_j
\phi_\alpha(\bsx_{h,j}) \right) \right] \I_{(b^m-1) \times 1 }
\\ && + \frac{1}{b^m} \left( \I_{(b^m-1) \times (b^m-1) } + \frac{b}{b-1} \gamma_d \Omega_p \right) \bsp_{d-1}
\\ & = & \left[-1+ \frac{1}{b^m} \prod^d_{j=1} \left( 1 + \frac{b}{b-1} \gamma_j
\phi_\alpha( \bsx_{0,j}) \right) \right] \I_{(b^m-1) \times 1 }
\\ && + \frac{1}{b^m} \sum^{b^m-1}_{h=1} \bsp_{d-1} (h) \I_{(b^m-1) \times 1} + \frac{1}{b^m} \frac{b}{b-1} \gamma_d \Omega_p \bsp_{d-1}.
\end{eqnarray*}
In the next lemma, we summarize an observation from \cite{D10}. Let
\begin{displaymath}
\Pi(g) = \left[ \Pi_{k,l} \right]_{\begin{subarray}{c} k=1,\dots,b^m-1 \\ l=1, \dots, b^m-1 \end{subarray} }
\end{displaymath}
where
\begin{equation} \label{eqdefpi}
\Pi_{k,l} = \left\{ \begin{array}{cc} 1 & \textrm{ if } \overline{k}(x) \equiv g^l(x) \pmod p \\ 0 & \textrm{ otherwise } \end{array} \right.
\end{equation}
and
\begin{displaymath}
\Pi(g^{-1}) = \left[ \Pi^{-1}_{k,l} \right]_{\begin{subarray}{c} k=1,\dots,b^m-1 \\ l=1, \dots, b^m-1 \end{subarray} }
\end{displaymath}
where
\begin{equation} \label{eqdefpiinv}
\Pi^{-1}_{k,l} = \left\{ \begin{array}{cc} 1 & \textrm{ if } \overline{k}(x) \equiv g^{-l} (x) \pmod p \\ 0 & \textrm{ otherwise, } \end{array} \right.
\end{equation}
be two permutation matrices, where $g$ is a primitive element which generates all elements of $( \zz_b[x] / p )^*= \left\{ g^0, g^1, \dots, g^{b^m-1} \right\}$; such an element $g$ is known to exist since the multiplicative group of every finite field is cyclic. Let $t_k = \deg (g^k \pmod p)$, $k=0,1,\dots,b^m-2$, and set
\begin{equation} \label{eqdefA3}
A_3 = \left[ b^{2 \alpha t_{ i - j \pmod{b^m-1}}} \right]_{ \begin{subarray}{c} i=1,\dots, b^m-1 \\ j=1,\dots, b^m-1 \end{subarray} }
\end{equation}
and note that $A_3$ is a circulant matrix, which allows us to use Fast Fourier Transforms (FFTs) as in \cite{NC05,NC06}.  We now state the lemma.

\begin{lemma} \label{lemJosi}
Let $p$ be an irreducible polynomial, let $g$ be a primitive element of $(\zz_b[x]/p)^\ast$, and let $\Pi(g)$, $\Pi(g^{-1})$, $A_3$ and $\Omega_p$ be defined as above. Then
\begin{displaymath}
 \Omega_p = \I_{(b^m-1) \times (b^m-1)} \frac{b-1}{b(b^{2 \alpha}-1)} - \frac{b^{2 \alpha +1} -1}{b(b^{2 \alpha}-1)} b^{- 2 \alpha m} \Pi(g) A_3 \Pi(g^{-1})^{\top} \, .
\end{displaymath}
\end{lemma}
\begin{proof}
It follows from the definition of $\phi_\alpha(x)$
\begin{displaymath}
\phi_\alpha\left( v_m \left( \frac{ \overline{h} \overline{q} }{p} \right)\right) = \frac{b-1}{b(b^{2 \alpha} -1)} - \frac{(b^{2 \alpha +1}-1) b^{-2 \alpha a_{0,h,q}}}{b(b^{2 \alpha}-1)} \, ,
\end{displaymath}
where $a_{0,h,q}$ denotes the smallest integer $a$ so that $\xi_{h,q,a} \neq 0$, and where
\begin{displaymath}
v_m \left( \frac{ \overline{h} \overline{q} }{q} \right) = \frac{ \xi_{h,q,1} }{b} + \frac{\xi_{h,q,2}}{b^2} + \dots.
\end{displaymath}
Hence
\begin{displaymath}
\Omega_p = \frac{b-1}{b(b^{2 \alpha}-1)} \I_{(b^m-1) \times (b^m-1)} - \frac{b^{2 \alpha +1}-1}{b(b^{2 \alpha}-1)} A_1,
\end{displaymath}
where \begin{displaymath}
A_1 = \left[ b^{- 2 \alpha a_{0,h,q} } \right]_{\begin{subarray}{c} q=1,\dots,b^m-1 \\ h=1,\dots,b^m-1 \end{subarray}}.
      \end{displaymath}
Now assume that for $w \in \zz_b[x]$ we have
\begin{equation} \label{eqdefu}
 \frac{w(x)}{p(x)} = u_{1,w} x^{-1} + u_{2,w} x^{-2} + \dots ,
\end{equation}
where $u_{j,w} \in \zz_b$. Then
\begin{displaymath}
v_m\left( \frac{ \overline{h} \overline{q} }{p} \right) = u_{1, \overline{h} \overline{q} } b^{-1} + u_{2, \overline{h} \overline{q}} b^{-2} + \dots + u_{m ,\overline{h} \overline{q}} b^{-m} \, ,
\end{displaymath}
hence, $a_{0,h,q}$ is the smallest integer $a$ so that $u_{a, \overline{h} \overline{q}} \neq 0$, $\overline{h}, \overline{q} \in \zz_b[x]$ (note that $p \not | \overline{h}, \overline{q}$).

The matrix $A_2$ given by
\begin{displaymath}
 A_2 = \Pi^{\top} (g) A_1 \Pi(g^{-1})
\end{displaymath}
is circulant. Indeed it can be checked that
\begin{equation} \label{eqdefA2}
A_2 = \left[ b^{ - 2 \alpha a_{0, g^{-j}, g^i}} \right]_{ \begin{subarray}{c} i=1,\dots,b^m-1 \\ j=1, \dots, b^m-1  \end{subarray} },
\end{equation}
where $g^{-j}$ and $g^{i}$ in Equation \eqref{eqdefA2} denote the integers associated with the polynomials $g^{-j} \pmod p$ and $g^{i} \pmod p$. We let $a_{0,g^{-j}, g^i} = r_{i-j}$ and note that $r_k = r_{k'}$ for $k \equiv k' \pmod{ b^m-1}$, as $g^{b^m-1} \equiv 1 \pmod{p}$, hence
\begin{displaymath}
A_2 = \left[ b^{-2 \alpha r_{i-j} } \right]_{ \begin{subarray}{c} i=1,\dots, b^m-1 \\ j=1, \dots, b^m-1 \end{subarray} } \, .
\end{displaymath}
The matrix $A_2$ is circulant and $r_k$ is the smallest integer $r$ such that $u_{r,g^k} \neq 0$, which implies using Equation \eqref{eqdefu} that
\begin{displaymath}
 \deg( g^k \pmod p) = \deg (p)-r_k \, ,
\end{displaymath}
and consequently
\begin{displaymath}
 r_k = m - \deg( g^k \pmod p) \, .
\end{displaymath}
Now denoting
\begin{displaymath}
 t_k = \deg( g^k \pmod p) \, ,
\end{displaymath}
we get
\begin{displaymath}
 A_2 = b^{-2 \alpha m} A_3 \, ,
\end{displaymath}
where $A_3$ is given by Equation \eqref{eqdefA3} and the result follows.
\end{proof}

Note that if the polynomial $p$ in the lemma above is primitive, then one can choose the primitive element $g(x) = x$. In Algorithm \ref{algfastcbcalg1} we show how to implement the CBC algorithm from Section \ref{seccbc}.
\begin{algorithm}[h!] \label{algfastcbc1}
\caption{Fast CBC algorithm}
\begin{algorithmic}[1]\label{algfastcbcalg1}
\REQUIRE $b$ a prime, $s, m \in \nn$ and weights $\bsgamma = (\gamma_j)_{j \geq 1}$.
\STATE Choose a primitive polynomial $p \in \zz_b[x]$, with $\deg(p)=m$, and choose $g(x) = x$.
\STATE $\bsmu_0 := \I_{(b^m-1 \times 1)}$.
\FOR{$d=1$ to $s$}
\STATE
$\tilde{B}_d  = A_3 \bsmu_{d-1}$.
\STATE $w_d = \arg \min_{w \in R_{b,m}} \tilde{B}_d(w) \, .$
 \STATE
\begin{displaymath}
\bsmu_d = \diag \left( \I_{1 \times (b^m-1)} \left( 1 + \frac{\gamma_d}{(b^{2 \alpha}-1)} \right)  - \gamma_d \frac{(b^{2 \alpha +1} -1) b^{-2 \alpha m}}{(b-1)(b^{2 \alpha}-1)}  A_3(w_d,:)  \right) \bsmu_{d-1} \, .
\end{displaymath}

\ENDFOR
\RETURN $\bsq=(q_1,\dots,q_s)$.
\end{algorithmic}
\end{algorithm}
Several remarks regarding Algorithm \ref{algfastcbcalg1} are in order.
\begin{remark} \label{rempermspace} As in \cite{NC05,NC06}, we search for the minimum in the permuted space, hence we minimize $\tilde{B}_d=\Pi^{\top}(g) B_{d-1}$. However, as in \cite{NC05,NC06}, the component $z_d$ can be found by mapping back $w_d$ using $\Pi(g)$.
\end{remark}

\begin{remark} \label{remmuupdate} We have $\bsmu_d = \Pi^{\top} (g^{-1}) \bsp_d$ and consequently update $\bsmu_d$ using $\bsmu_{d-1}$. Hence we do not need to permute back and forth, but can complete the algorithm in the permuted space.
\end{remark}

The next corollary gives information on the computational complexity of Algorithm \ref{algfastcbcalg1}. We use
\begin{equation} \label{eqdefa}
 \bsa = \left[ \begin{array}{c}
 b^{2 \alpha t_0} \\ b^{2 \alpha t_1} \\ \vdots \\ b^{2 \alpha t_{b^m-2}} \, .
\end{array}
 \right] \,
\end{equation}
to denote the vector generating the circulant matrix $A_3$ in Lemma \ref{lemJosi} and Algorithm \ref{algfastcbcalg1}.
\begin{corollary} \label{corrcompcomplalgfastcbc} Assume that the vector $\bsa$ in Equation \eqref{eqdefa} has been precomputed and stored using $\mathcal{O}(b^m)$ memory. Then Algorithm \ref{algfastcbcalg1} can be completed in time $\mathcal{O}(s b^m m)$ and memory $\mathcal{O}(b^m)$.
\end{corollary}

For a proof, see \cite{NC05,NC06} or also \cite[Section~10.3]{DP09}.

\section{Numerical Experiments} \label{secnumexp}

In this section, we numerically investigate the performance of the
CBC algorithm presented in Section \ref{seccbc}; we rely on Section
\ref{secimplementation} for the implementation of the algorithm. In
Tables \ref{table1} - \ref{table3}, we present values of $B(\bsq,
\alpha,\bsgamma)$ for different choices of $\alpha$ and $\bsgamma$,
where $\bsq$ is constructed using Algorithm \ref{algfastcbcalg1}.

We compare the performance of the CBC algorithm to the performance of digital nets. As was done with scrambled polynomial lattice rules in Section \ref{secvarscrambledpolyrules}, we can study the variance of the estimator $\hat{I}(f)$ given in Equation \eqref{eqdefestimator}, consider the worst-case variance of multivariate integration in $V_{\alpha,s,\bsgamma}$ and bound this variance as follows:
\begin{equation}\label{eqbounddignets}
\Var(Q_{b^m,s}(C_1,\dots,C_s), V_{\alpha,s,\bsgamma} )
\leq \sum_{\bsk \in \Dcal(C_1,\ldots, C_s) \setminus \{\bs0\}} r_{2\alpha+1, \bsgamma}(\bsk),
\end{equation}
where $C_1,\dots,C_s$ are the generating matrices of the digital net
under consideration and $\Dcal(C_1,\ldots, C_s)$ is its dual space. We denote the bound \eqref{eqbounddignets} by $B((C_1,\dots,C_s), \alpha,\bsgamma)$, and remark that $B((C_1,\dots,C_s),\alpha,\bsgamma)$ can also be
computed using Equation \eqref{eqformulabound}, where $\left\{
\bsx_h \right\}^{b^m-1}_{h=0}$ is the digital net generated by
$C_1,\dots,C_s$.

Consequently, we compare the values of $B(\bsq, \alpha,\bsgamma)$ to
the values of $B((C_1,\dots,C_s), \alpha,\bsgamma)$ in Tables
\ref{table1} - \ref{table3}; in each cell, the top number
corresponds to the CBC construction and the bottom one to the digital
net. We choose the following digital nets: For $s=1$, we simply
choose equidistributed points, $x_h=\frac{h}{b^m}$,
$h=0,\dots,b^m-1$, for $s=5$, we use Pirsic's implementation of
Niederreiter-Xing points, \cite{P02}, and for $s=50$ and $s=100$, we
use Sobol points as constructed in \cite{JK03}; we point out that
for the CBC construction, we choose $b=2$ and likewise, the digital
nets under consideration are digital nets over $\zz_2$.

We derive the following conclusions from the tables: For $s=1$, as
expected, we obtain the optimal rate of convergence, $2^{-(2
\alpha+1)m}$, and observe the same values for the CBC construction
as for the digital nets. Regarding the case $s=5$, the values are
comparable, however, the Niederreiter-Xing construction seems to be
slightly better than the CBC construction for the examples
considered. Finally, for $s=50$ and $s=100$, the performances of the
two methods are again comparable, however, this time, the CBC
construction seems to outperform the digital nets.

\begin{table}
\begin{center}

\begin{tabular}{|c|c|c|c|c|c|c|c|c|}
\hline
& \multicolumn{4}{|c|}{$\alpha=0.5$} & \multicolumn{4}{|c|}{$\alpha=1$} \\
\hline
$m=$ & $s=1$ & $s=5$ & $s=50$ & $s=100$ & $s=1$ & $s=5$ & $s=50$ & $s=100$ \\
\hline

\multirow{2}{*}{$4$} &3.91e-03&1.46e+00&7.04e+13&7.92e+28&8.14e-05&4.37e-02&1.10e+05&1.95e+11\\
 &3.91e-03&1.48e+00&7.04e+13&7.92e+28&8.14e-05&4.90e-02&1.10e+05&1.95e+11\\\hline
\multirow{2}{*}{$5$} &9.77e-04&6.16e-01&3.52e+13&3.96e+28&1.02e-05&1.09e-02&5.52e+04&9.74e+10\\
 &9.77e-04&6.34e-01&3.52e+13&3.96e+28&1.02e-05&1.32e-02&5.52e+04&9.74e+10\\\hline
\multirow{2}{*}{$6$} &2.44e-04&2.66e-01&1.76e+13&1.98e+28&1.27e-06&3.45e-03&2.76e+04&4.87e+10\\
 &2.44e-04&2.61e-01&1.76e+13&1.98e+28&1.27e-06&3.17e-03&2.76e+04&4.87e+10\\\hline
\multirow{2}{*}{$7$} &6.10e-05&1.08e-01&8.80e+12&9.90e+27&1.59e-07&9.05e-04&1.38e+04&2.44e+10\\
 &6.10e-05&1.04e-01&8.80e+12&9.90e+27&1.59e-07&7.19e-04&1.38e+04&2.44e+10\\\hline
\multirow{2}{*}{$8$} &1.53e-05&4.24e-02&4.40e+12&4.95e+27&1.99e-08&2.36e-04&6.90e+03&1.22e+10\\
 &1.53e-05&3.93e-02&4.40e+12&4.95e+27&1.99e-08&1.48e-04&6.90e+03&1.22e+10\\\hline
\multirow{2}{*}{$9$} &3.81e-06&1.74e-02&2.20e+12&2.48e+27&2.48e-09&6.10e-05&3.45e+03&6.09e+09\\
 &3.81e-06&1.44e-02&2.20e+12&2.48e+27&2.48e-09&2.86e-05&3.45e+03&6.09e+09\\\hline
\multirow{2}{*}{$10$} &9.54e-07&6.41e-03&1.10e+12&1.24e+27&3.10e-10&1.29e-05&1.72e+03&3.04e+09\\
 &9.54e-07&5.21e-03&1.10e+12&1.24e+27&3.10e-10&5.56e-06&1.72e+03&3.04e+09\\\hline
\multirow{2}{*}{$11$} &2.38e-07&2.29e-03&5.50e+11&6.19e+26&3.88e-11&2.56e-06&8.62e+02&1.52e+09\\
 &2.38e-07&1.82e-03&5.50e+11&6.19e+26&3.88e-11&1.01e-06&8.62e+02&1.52e+09\\\hline
\multirow{2}{*}{$12$} &5.96e-08&8.39e-04&2.75e+11&3.09e+26&4.85e-12&5.03e-07&4.31e+02&7.61e+08\\
 &5.96e-08&6.17e-04&2.75e+11&3.09e+26&4.85e-12&1.78e-07&4.31e+02&7.61e+08\\\hline
\multirow{2}{*}{$13$} &1.49e-08&3.09e-04&1.37e+11&1.55e+26&6.06e-13&1.05e-07&2.15e+02&3.81e+08\\
 &1.49e-08&2.06e-04&1.37e+11&1.55e+26&6.06e-13&3.07e-08&2.16e+02&3.81e+08\\\hline
\multirow{2}{*}{$14$} &3.73e-09&1.12e-04&6.87e+10&7.74e+25&7.58e-14&2.56e-08&1.08e+02&1.90e+08\\
 &3.73e-09&6.76e-05&6.87e+10&7.74e+25&7.57e-14&5.17e-09&1.08e+02&1.90e+08\\\hline
\multirow{2}{*}{$15$} &9.31e-10&3.66e-05&3.44e+10&3.87e+25&9.27e-15&4.98e-09&5.38e+01&9.52e+07\\
 &9.31e-10&2.18e-05&3.44e+10&3.87e+25&9.33e-15&8.54e-10&5.39e+01&9.52e+07\\\hline
\multirow{2}{*}{$16$} &2.33e-10&1.29e-05&1.72e+10&1.93e+25&1.22e-15&8.92e-10&2.69e+01&4.76e+07\\
 &2.33e-10&6.94e-06&1.72e+10&1.93e+25&1.11e-15&1.38e-10&2.69e+01&4.76e+07\\\hline
\end{tabular}

\end{center}
 \caption{Values of $B(\bsq, \alpha, \bsgamma)$ and $B((C_1,\dots,C_s), \alpha,\bsgamma)$ for $\gamma_j=1$, $j=1,\dots,s$ and $\bsq$ constructed using the CBC algorithm; the top number gives the value of $B(\bsq, \alpha,\bsgamma)$, the bottom the value of $B((C_1, \dots, C_s), \alpha,\bsgamma)$.}
\label{table1}
\end{table}

\begin{table}
 \begin{center}
  \begin{tabular}{|c|c|c|c|c|c|c|c|c|}
\hline
& \multicolumn{4}{|c|}{$\alpha=0.5$} & \multicolumn{4}{|c|}{$\alpha=1$} \\
\hline
$m=$ & $s=1$ & $s=5$ & $s=50$ & $s=100$ & $s=1$ & $s=5$ & $s=50$ & $s=100$ \\
\hline
\multirow{2}{*}{$4$} &3.42e-03&4.64e-01&2.03e+01&2.04e+01&7.12e-05&1.47e-02&2.82e-01&2.83e-01\\
 &3.42e-03&4.84e-01&2.04e+01&2.06e+01&7.12e-05&1.83e-02&3.46e-01&3.48e-01\\\hline
\multirow{2}{*}{$5$} &8.54e-04&1.87e-01&9.99e+00&1.01e+01&8.90e-06&3.54e-03&1.16e-01&1.17e-01\\
 &8.54e-04&1.95e-01&1.01e+01&1.01e+01&8.90e-06&4.45e-03&1.38e-01&1.39e-01\\\hline
\multirow{2}{*}{$6$}&2.14e-04&7.75e-02&4.91e+00&4.95e+00&1.11e-06&1.04e-03&4.78e-02&4.82e-02\\
 &2.14e-04&7.46e-02&4.94e+00&4.98e+00&1.11e-06&9.29e-04&5.30e-02&5.34e-02\\\hline
\multirow{2}{*}{$7$} &5.34e-05&2.96e-02&2.40e+00&2.42e+00&1.39e-07&2.54e-04&1.85e-02&1.87e-02\\
 &5.34e-05&2.80e-02&2.44e+00&2.47e+00&1.39e-07&1.97e-04&2.39e-02&2.41e-02\\\hline
\multirow{2}{*}{$8$} &1.34e-05&1.17e-02&1.17e+00&1.18e+00&1.74e-08&5.77e-05&7.39e-03&7.45e-03\\
 &1.34e-05&1.01e-02&1.20e+00&1.21e+00&1.74e-08&3.79e-05&1.08e-02&1.09e-02\\\hline
\multirow{2}{*}{$9$} &3.34e-06&4.43e-03&5.66e-01&5.71e-01&2.17e-09&1.29e-05&2.84e-03&2.87e-03\\
 &3.34e-06&3.54e-03&5.89e-01&5.95e-01&2.17e-09&6.98e-06&4.94e-03&4.97e-03\\\hline
\multirow{2}{*}{$10$} &8.34e-07&1.56e-03&2.72e-01&2.75e-01&2.72e-10&3.03e-06&1.08e-03&1.09e-03\\
 &8.34e-07&1.22e-03&2.88e-01&2.90e-01&2.72e-10&1.28e-06&2.24e-03&2.26e-03\\\hline
\multirow{2}{*}{$11$} &2.09e-07&5.45e-04&1.30e-01&1.31e-01&3.40e-11&6.24e-07&4.01e-04&4.06e-04\\
 &2.09e-07&4.10e-04&1.42e-01&1.43e-01&3.40e-11&2.22e-07&9.58e-04&9.66e-04\\\hline
\multirow{2}{*}{$12$} &5.22e-08&1.93e-04&6.20e-02&6.26e-02&4.24e-12&1.16e-07&1.49e-04&1.51e-04\\
 &5.22e-08&1.35e-04&6.73e-02&6.80e-02&4.24e-12&3.79e-08&3.59e-04&3.64e-04\\\hline
\multirow{2}{*}{$13$} &1.30e-08&7.07e-05&2.94e-02&2.97e-02&5.31e-13&2.48e-08&5.43e-05&5.51e-05\\
 &1.30e-08&4.38e-05&3.33e-02&3.36e-02&5.30e-13&6.34e-09&2.11e-04&2.13e-04\\\hline
\multirow{2}{*}{$14$} &3.26e-09&2.27e-05&1.39e-02&1.40e-02&6.62e-14&4.56e-09&1.99e-05&2.02e-05\\
 &3.26e-09&1.40e-05&1.58e-02&1.60e-02&6.62e-14&1.04e-09&8.23e-05&8.31e-05\\\hline
\multirow{2}{*}{$15$} &8.15e-10&8.01e-06&6.49e-03&6.57e-03&8.49e-15&1.10e-09&7.15e-06&7.26e-06\\
 &8.15e-10&4.41e-06&7.70e-03&7.78e-03&8.22e-15&1.67e-10&4.44e-05&4.48e-05\\\hline
\multirow{2}{*}{$16$} &2.04e-10&2.68e-06&3.02e-03&3.06e-03&9.99e-16&1.81e-10&2.57e-06&2.61e-06\\
 &2.04e-10&1.37e-06&3.76e-03&3.80e-03&8.88e-16&2.64e-11&2.14e-05&2.15e-05\\\hline
\end{tabular}
 \end{center}

\caption{Values of $B(\bsq, \alpha, \bsgamma)$ and $B((C_1,\dots,C_s), \alpha,\bsgamma)$ for $\gamma_j=0.875^j$, $j=1,\dots,s$ and $\bsq$ constructed using the CBC algorithm; the top number gives the value of $B(\bsq, \alpha,\bsgamma)$, the bottom the value of $B((C_1, \dots, C_s), \alpha,\bsgamma)$.}
\label{table2}
\end{table}

\begin{table}
 \begin{center}
\begin{tabular}{|c|c|c|c|c|c|c|c|c|}
\hline
& \multicolumn{4}{|c|}{$\alpha=0.5$} & \multicolumn{4}{|c|}{$\alpha=1$} \\
\hline
$m=$ & $s=1$ & $s=5$ & $s=50$ & $s=100$ & $s=1$ & $s=5$ & $s=50$ & $s=100$ \\
\hline
\multirow{2}{*}{$4$} &3.91e-03&2.75e-02&4.75e-02&4.90e-02&8.14e-05&7.68e-04&1.80e-03&1.88e-03\\
 &3.91e-03&3.20e-02&5.97e-02&6.17e-02&8.14e-05&1.27e-03&4.40e-03&4.62e-03\\\hline
\multirow{2}{*}{$5$} &9.77e-04&8.98e-03&1.78e-02&1.84e-02&1.02e-05&1.48e-04&4.88e-04&5.20e-04\\
 &9.77e-04&1.25e-02&2.22e-02&2.30e-02&1.02e-05&3.13e-04&1.23e-03&1.30e-03\\\hline
\multirow{2}{*}{$6$} &2.44e-04&2.95e-03&6.29e-03&6.56e-03&1.27e-06&3.37e-05&1.20e-04&1.31e-04\\
 &2.44e-04&3.20e-03&7.23e-03&7.66e-03&1.27e-06&3.62e-05&2.47e-04&2.89e-04\\\hline
\multirow{2}{*}{$7$} &6.10e-05&8.96e-04&2.22e-03&2.35e-03&1.59e-07&5.05e-06&2.91e-05&3.31e-05\\
 &6.10e-05&1.11e-03&2.65e-03&2.88e-03&1.59e-07&8.38e-06&6.91e-05&9.12e-05\\\hline
\multirow{2}{*}{$8$} &1.53e-05&2.96e-04&7.86e-04&8.36e-04&1.99e-08&1.04e-06&6.94e-06&8.16e-06\\
 &1.53e-05&3.44e-04&1.00e-03&1.12e-03&1.99e-08&1.37e-06&2.16e-05&3.52e-05\\\hline
\multirow{2}{*}{$9$} &3.81e-06&9.34e-05&2.81e-04&3.02e-04&2.48e-09&1.90e-07&1.67e-06&2.02e-06\\
 &3.81e-06&9.15e-05&3.27e-04&3.64e-04&2.48e-09&1.60e-07&4.70e-06&6.18e-06\\\hline
\multirow{2}{*}{$10$} &9.54e-07&2.78e-05&9.54e-05&1.03e-04&3.10e-10&4.07e-08&4.20e-07&5.14e-07\\
 &9.54e-07&2.69e-05&1.19e-04&1.32e-04&3.10e-10&2.49e-08&1.83e-06&2.37e-06\\\hline
\multirow{2}{*}{$11$} &2.38e-07&8.95e-06&3.34e-05&3.64e-05&3.88e-11&6.34e-09&9.59e-08&1.21e-07\\
 &2.38e-07&8.25e-06&4.10e-05&4.71e-05&3.88e-11&3.76e-09&3.25e-07&5.44e-07\\\hline
\multirow{2}{*}{$12$} &5.96e-08&2.68e-06&1.18e-05&1.30e-05&4.85e-12&1.30e-09&2.36e-08&3.03e-08\\
 &5.96e-08&2.54e-06&1.46e-05&1.68e-05&4.85e-12&6.34e-10&1.16e-07&1.64e-07\\\hline
\multirow{2}{*}{$13$} &1.49e-08&8.29e-07&4.05e-06&4.50e-06&6.06e-13&2.04e-10&5.79e-09&7.61e-09\\
 &1.49e-08&7.08e-07&4.69e-06&5.76e-06&6.06e-13&9.21e-11&2.45e-08&5.16e-08\\\hline
\multirow{2}{*}{$14$} &3.73e-09&2.50e-07&1.40e-06&1.57e-06&7.58e-14&4.11e-11&1.40e-09&1.88e-09\\
 &3.73e-09&1.97e-07&1.60e-06&1.98e-06&7.57e-14&1.29e-11&7.04e-09&2.04e-08\\\hline
\multirow{2}{*}{$15$} &9.31e-10&7.70e-08&4.91e-07&5.54e-07&9.27e-15&6.15e-12&3.45e-10&4.79e-10\\
 &9.31e-10&5.59e-08&5.50e-07&7.09e-07&9.33e-15&1.74e-12&2.52e-09&4.52e-09\\\hline
\multirow{2}{*}{$16$} &2.33e-10&2.34e-08&1.71e-07&1.95e-07&1.22e-15&1.00e-12&8.51e-11&1.22e-10\\
 &2.33e-10&1.69e-08&1.89e-07&2.52e-07&1.11e-15&2.70e-13&9.54e-10&1.45e-09\\\hline
\end{tabular}
  \end{center}
\caption{Values of $B(\bsq, \alpha, \bsgamma)$ and $B((C_1,\dots,C_s), \alpha,\bsgamma)$ for $\gamma_j=j^{-2}$, $j=1,\dots,s$ and $\bsq$ constructed using the CBC algorithm; the top number gives the value of $B(\bsq, \alpha,\bsgamma)$, the bottom the value of $B((C_1, \dots, C_s), \alpha,\bsgamma)$.}
\label{table3}
\end{table}

\noindent{\bf Author's Addresses:}\\

\noindent Jan Baldeaux, School of Mathematics and Statistics, The University of New South Wales, Sydney 2052, Australia. Email: Jan.Baldeaux@unsw.edu.au\\

\noindent Josef Dick, School of Mathematics and Statistics, The University of New South Wales, Sydney 2052, Australia. Email: josef.dick@unsw.edu.au\\

\end{document}